\documentclass[11pt]{amsart}   %%%%%%%Please do not change
\usepackage{mathtext}
\usepackage{graphics} %%%%%%% This package is needed for graphics
\usepackage{amscd}    %%%%%%% This package is needed for commutative diagrams.
\usepackage{amssymb}

\usepackage{graphicx}

\newcommand{\sysn}{\left\{\begin{array}{rcl}}
\newcommand{\sysk}{\end{array}\right.}

\newcommand{\ingrw}[2]{\includegraphics[width=#1mm]{#2}}

\newtheorem{theorem}{Theorem}[section]

\theoremstyle{definition}
\newtheorem{definition}[theorem]{Definition}
\theoremstyle{remark}
\newtheorem{remark}[theorem]{Remark}
\numberwithin{equation}{section}
\newtheorem{corollary}[theorem]{Corollary}

\begin{document}

\vspace{0.5in}

\title[Some properties of minimal $S(\alpha)$ and $S(\alpha)FC$ spaces]%
{Some properties of minimal $S(\alpha)$ and $S(\alpha)FC$ spaces}

%    Information for first author:
\author{Alexander V. Osipov}
\address{Institute of Mathematics and Mechanics, Ural Branch of the Russian Academy of Sciences, Ural Federal University, Ekaterinburg, Russia}
%    Current address (if needed):
%\curraddr{}
\email{OAB@list.ru}

%    Information for second author (if needed):
%\author{Author Two}
%\address{}
%\email{}

%    General info

%%%%%%%%%%%%%%%%%%%%%%%%%%%%%%%%%%%%%%%%%%%%%%%%%%%
\subjclass[2010]{Primary 54D25, 54B10, 54D45; Secondary 54A10, 54D30}                                    %
%                                                                                                                           %
%         Please use the current 2010 Mathematics Subject Classification:             %
%         http://www.ams.org/mathscinet/msc/                                                        %
%         http://www.zentralblatt-math.org/msc/en/                                                 %
%%%%%%%%%%%%%%%%%%%%%%%%%%%%%%%%%%%%%%%%%%%%%%%%%%%

\keywords{$\theta^n$-closure, $S(n)$-closed space,
$S(n)$-$\theta$-closed space, minimal $S(n)$ space}

\begin{abstract}
A $S(n)$-space is $S(n)$-functionally compact ($S(n)FC$) if every
continuous function onto a $S(n)$-space is closed. $S(n)$-closed,
$S(n)$-$\theta$-closed, minimal $S(n)$ and $S(n)FC$ spaces are
characterized in terms of $\theta(n)$-complete accumulation
points. In paper we also give new characteristics of $R$-closed
and regular functionally compact spaces.

Results obtained to answer some questions raised by D.Dikranjan,
E.Giuli, L.Friedler, M.Girou, D.Pettey and J.Porter.
\end{abstract}

\maketitle

\section{Introduction}

  Dikranjan and Giuli~{\cite{dg}} introduced a notion of the $\theta^n$-closure operator
   and developed a theory of  $S(n)$-closed and $S(n)$-$\theta$-closed spaces.
   Jiang, Reilly, and Wang ~{\cite{jrw}} used the $\theta^n$-closure
in studying properties of minimal $S(n)$-spaces.

In work ~{\cite{osip5}} continues the study of properties inherent
in $S(n)$-closed and $S(n)$-$\theta$-closed spaces, using the
$\theta^n$-closure operator; in addition, wider classes of spaces
(weakly $S(n)$-closed and weakly $S(n)$-$\theta$-closed spaces)
are introduced.

 In this paper we continue the investigation of $S(n)$-closed,
 $S(n)$-$\theta$-closed, minimal $S(n)$ spaces with the use of $\theta(n)$-complete accumulation points. As we introduce
 new classes of $S(n)$-spaces --- $S(n)$-functionally compact spaces and to answer some questions raised
 in ~{\cite{dg,jrw}}.

 Section $2$ acquaints the reader with main
definitions and known properties in the theory of $S(n)$-spaces.
Section $3$ is completely devoted to the study of weakly
$S(n)$-closed and weakly $S(n)$-$\theta$-closed spaces. It is
proved that any $S(n)$-closed ($S(n)$-$\theta$-closed) space is
weakly $S(n)$-closed (weakly $S(n)$-$\theta$-closed). In the
remaining sections, we characterize $S(n)$-closed,
$S(n)$-$\theta$-closed, minimal $S(n)$ spaces, $S(n)$-functionally
compact, $R$-closed, minimal regular and regular functionally
compact spaces with the use of $\theta(n)$-complete accumulation
and $\theta(\omega)$-complete accumulation points.

\section{Main definitions and notation}

Let $X$ be a topological space, $M\subseteq X$, and $x\in X$. For
any $n\in \mathbb N$, we consider the $\theta^{n}$-closure
operator: $x\notin cl_{\theta^n} M$ if there exists a set of open
neighborhoods $U_{1}$, $U_{2}$, ..., $U_{n}$ of the point $x$ such
that $clU_{i}\subseteq U_{i+1}$ for $i=1, 2, ..., n-1$ and
$clU_{n}\bigcap M=\emptyset$  if $n>1$; $cl_{\theta^0}M=clM$ if
$n=0$; and, for $n=1$, we get the $\theta$-closure operator, i.e.,
$cl_{\theta^1}M=cl_{\theta}M$. A set $M$ is $\theta^n$-closed if
$M=cl_{\theta^n}M$. Denote by $Int_{\theta^n}M=X\setminus
cl_{\theta^n}(X\setminus M)$ the $\theta^n$-interior of the set
$M$. Evidently, $cl_{\theta^n}(cl_{\theta^s}M)=cl_{\theta^{n+s}}M$
for $M\subseteq X$ and $n,s\in \mathbb N$. For $n\in \mathbb N$
and a filter $\mathcal F$ on $X$, denote by
$ad_{\theta^{n}}\mathcal F$ the set of $\theta^n$-adherent points,
i.e., $ad_{\theta^n}\mathcal F=\bigcap \{ cl_{\theta^n}\mathcal
F_{\alpha} : F_{\alpha}\in \mathcal F\}$. In particular,
$ad_{\theta^0}\mathcal F=ad\mathcal F$ is the set of adherent
points of the filter $\mathcal F$. For any $n\in \mathbb N$, a
point $x\in X$ is $S(n)$-separated from a subset $M$ if $x\notin
cl_{\theta^n}M$. For example, $x$ is $S(0)$-separated from $M$ if
$x\notin clM$. For $n>0$, the relation of $S(n)$-separability of
points is symmetric. On the other hand, $S(0)$-separability may be
not symmetric in some not $T_1$-spaces. Therefore, we say that
points $x$ and $y$ are $S(0)$-separated if $x\notin
\{\overline{y}\}$ and $y\notin \{\overline{x}\}$.

Let $n\in \mathbb N$ and $X$ be a topological space.

1. $X$ is called an $S(n)$-space if any two distinct points of $X$
are $S(n)$-separated.

2. A filter $\mathcal F$ on $X$ is called an $S(n)$-filter if
every point, not being an adherent point of the filter $\mathcal
F$, is $S(n)$-separated from some element of the filter $\mathcal
F$.

3. An open cover $\{U_{\alpha}\}$ of the space $X$ is called an
$S(n)$-cover if every point of $X$ lies in the $\theta^n$-interior
of some $U_{\alpha}$.

It is obvious that $S(0)$-spaces are $T_0$-spaces, $S(1)$-spaces
are Hausdorff spaces, and $S(2)$-spaces are Urysohn spaces. It is
clear that every filter is an $S(0)$-filter, every open cover is
an $S(0)$-cover, and every open filter is an $S(1)$-filter. Open
$S(2)$-filters are called Urysohn filters. For $n > 1$, open
$S(n)$-filters were defined in ~{\cite{pv}}. $S(1)$-covers are
called Urysohn covers. In a regular space, every filter (every
cover) is an $S(n)$-filter ($S(n)$-cover) for any $n\in \mathbb
N$.

$S(n)$-closed and $S(n)$-$\theta$-closed spaces are $S(n)$-spaces,
closed and, respectively, $\theta$-closed in any $S(n)$-space
containing them.

Porter and Votaw ~{\cite{pv}} characterized $S(n)$-closed spaces
by means of open $S(n)$-filters and $S(n-1)$-covers.

Let $n\in \mathbb N^{+}$ and $X$ be an $S(n)$-space. Then the
following conditions are equivalent:

(1) $ad_{\theta^n}\mathcal F\neq \emptyset$ for any open filter
$\mathcal F$ on $X$;

(2) $ad\mathcal F\neq \emptyset$ for any open $S(n)$-filter
$\mathcal F$ on $X$;

(3) for any $S(n-1)$-cover $\{U_{\alpha}\}$ of the space $X$ there
exist $\alpha_1, \alpha_2,...,\alpha_k$ such that
$X=\bigcup_{i=1}^{k} \overline{U_{\alpha_i}}$;

(4) $X$ is an $S(n)$-closed space.

Dikranjan and Giuli  ~{\cite{dg}} characterized
$S(n)$-$\theta$-closed spaces in terms of $S(n-1)$-filters and
$S(n-1)$-covers.

Let $n\in N^+$ and $X$ be an $S(n)$-space. Then the following
conditions are equivalent:

(1) $ad\mathcal F\neq \emptyset$ for any closed $S(n-1)$-filter
$\mathcal F$ on $X$;

(2) any $S(n-1)$-cover of $X$ has a finite subcover;

(3) $ad_{\theta^{(n-1)}}\mathcal F\neq \emptyset$ for any closed
filter $\mathcal F$ on $X$;

(4) $X$ is an $S(n)$-$\theta$-closed space.

 Note that, for $n=1$, $S(1)$-closedness and $S(1)$-$\theta$-closedness are $H$-closedness and
compactness, respectively. For $n=2$, $S(2)$-closedness and
$S(2)$-$\theta$-closedness are $U$-closedness and
$U$-$\theta$-closedness, respectively. From characteristics
themselves, it follows that any $S(n)$-$\theta$-closed subspace of
an $S(n)$-space is an $S(n)$-closed space.

Recall that a open cover $\mathcal V$ is a shrinkable refinement
of open cover $\mathcal U$ if and only if for each $V\in \mathcal
V$, there is a $U\in \mathcal U$ such that $\overline{V}\subseteq
U$. A open cover $\mathcal V$ is a regular refinement of $\mathcal
U$ if and only if $\mathcal V$ refines $\mathcal U$ is a
shrinkable refinement of itself. An open cover is regular if and
only if it has an open refinement.

 An open filter base $\mathcal F$ in $X$ is a regular filter base
 if and only if for each $U\in \mathcal F$, there exists $V\in
 \mathcal F$ such that $\overline{V}\subseteq U$.

A $R$-closed space is a regular space closed in any regular space
containing them.

Berri, Sorgenfrey ~{\cite{bs}} characterized $R$-closed spaces by
means of regular filters and regular covers.

Let $X$ be a regular space. The following are equivalent:

(1) $X$ is $R$-closed.

(2) Every regular filter base in $X$ is fixed.

(3) Every regular cover has a finite subcover.

For undefined notions and related theorems, we refer readers to
~{\cite{dg}}.

\section{Weakly $S(n)$-closed and weakly $S(n)$-$\theta$-closed spaces}

In the Aleksandrov and Urysohn memoir on compact spaces
~{\cite{alur}}, the notion of a $\theta$-complete accumulation
point was introduced. A point $x$ is called a $\theta$-complete
accumulation point of a set $F$ if $|F\bigcap \overline{U}|=|F|$
for any neighborhood $U$ of the point $x$. It was noted that any
$H$-closed space has the following property:

(*) any infinite set of regular power has a $\theta$-complete
accumulation point. However, the converse is not true. The first
example of a space possessing property (*) and not being
$H$-closed was constructed by Kirtadze ~{\cite{kir}}. Simple
examples in~{\cite{osip5,por}} also shows the converse is not
true.

{\bf Example 1.} (Example 1 in \cite{osip5}).\ \ Let $T_{1}$ and
$T_{2}$ be two copies of the Tychonoff plane $T= ((\omega_{1} +
1)\times(\omega_{0} + 1)\setminus\{\omega_1,\omega_0\}$, whose
elements will be denoted by $(\alpha,n,1)$ and $(\alpha, n, 2)$,
respectively. On the topological sum $T_{1}\oplus T_{2}$, we
consider the identifications $(\omega_1,k, 1)\sim (\omega_{1}, 2k,
2)$ for every $k\in \mathbb N$; and we identify all points
$(\omega_1, 2k-1,2)$ for any $k\in \mathbb N$ with the same point
$b$. Adding, to the obtained space, a point $a$ with the base of
neighborhoods $U_{\beta,k}(a)=\{a\}\bigcup\{(\alpha,n,1): \beta <
\alpha < \omega_1, k < n\leq \omega_0 \}$ for arbitrary $\beta <
\omega_1$ and $k<\omega_0$, we get a Urysohn space $X$.

Note that space $X$ is an example of a non-$H$-closed, Urysohn
space with the property that for every chain of non-empty sets,
the intersection of the $\theta$-closures of the sets is nonempty,
every infinite set has a $\theta$-complete accumulation point.
J.Porter investigated the space with the same properties in
\cite{por}.

\bigskip

\begin{center}
\ingrw{110}{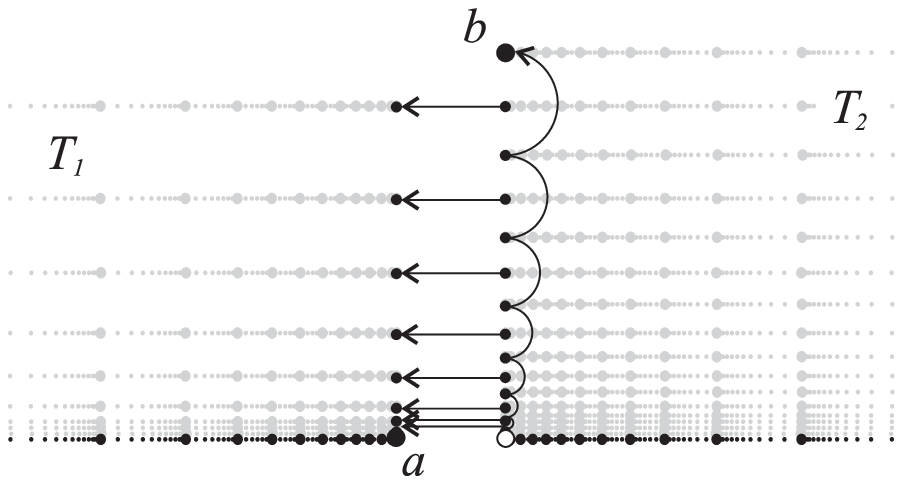}

\end{center}

\bigskip

\begin{definition} A neighborhood $U$ of a set $A$ is called an $n$-hull of
the set $A$ if there exists a set of neighborhoods $U_1$, $U_2$,
..., $U_n=U$ of the set $A$ such that $clU_{i}\subseteq U_{i+1}$
for $i=1,...,n-1$.
\end{definition}

\begin{definition} A point $x$ from $X$ is called a
$\theta^{0}(n)$-complete accumulation ($\theta(n)$-complete
accumulation) point of an infinite set $F$ if $|F\bigcap U|=|F|$
($|F\bigcap \overline{U}|=|F|$) for any $U$, where $U$ is an
$n$-hull of the point $x$.
\end{definition}

 Note that, for $n=1$, a $\theta^{0}(1)$-complete accumulation point is a point of complete accumulation,
and a $\theta(1)$-complete accumulation point is a
$\theta$-complete accumulation point.

\begin{definition} A topological space $X$ is called weakly
$S(n)$-$\theta$-closed (weakly $S(n)$-closed) if any infinite set
of regular power of the space $X$ has a $\theta^0(n)$-complete
accumulation ($\theta(n)$-complete accumulation) point.
\end{definition}

 Note that any $\theta^0(n)$-complete accumulation point is a
$\theta(n)$-complete accumulation point; hence, any weakly
$S(n)$-$\theta$-closed space is weakly $S(n)$-closed. Moreover,
since a $\theta(n)$-complete accumulation point is a
$\theta^{0}(n+1)$-complete accumulation point, it follows that a
weakly $S(n)$-closed space will be weakly
$S(n+1)$-$\theta$-closed. For $n=1$, weakly $S(1)$-$\theta$-closed
and weakly $S(1)$-closed spaces are compact Hausdorff spaces and
spaces with property (*), respectively.

\begin{theorem}
\label{tm1}

Let $X$ be an $S(n)$-closed $S(n)$-space. Then $X$ is weakly
$S(n)$-closed.
\end{theorem}

\begin{proof}

 Suppose the contrary. Let $X$ be $S(n)$-closed
but not weakly $S(n)$-closed. Then in the spaces $X$ there exists
an infinite set $F$ of regular power that has no
$\theta(n)$-complete accumulation point. For any point $x\in X$,
there exists an $n$-hull $U$ of the point $x$ with the property
$|F\bigcap \overline{U}|<|F|$. If we take such $n$-hull for every
point $x\in X$, we derive an $S(n-1)$-cover of the space $X$. By
$S(n)$-closedness, there exists a finite family $\mathcal U$ of
$n$-hulls such that $|F\bigcap \overline{U}|<|F|$ $\forall U \in
\mathcal U$ and $\bigcup \overline{\mathcal U}=X$. This
contradicts the fact that $F$ is infinite set of regular power.

\end{proof}

\begin{theorem}
\label{tm2}

Let $X$ be an $S(n)$-$\theta$-closed $S(n)$-space. Then $X$ is
weakly $S(n)$-$\theta$-closed space.
\end{theorem}

A proof of Theorem~\ref{tm2} is analogous to that of Theorem
~\ref{tm1}.

 It was proved in ~{\cite{dg}} that $S(n)$-closedness
implies $S(n+1)$-$\theta$-closedness. Thus, for $S(n)$-spaces,
classes of the considered spaces are presented in the following
diagram:

\bigskip

\begin{tabular}{ccc}
compact Hausdorff space &  $\Longleftrightarrow$ &     weakly
$S(1)$-$\theta$-closed\\
$\Downarrow$&&                                     $\Downarrow$\\
 $H$-closed &  $\Longrightarrow$ &     weakly $H$-closed\\
 $\Downarrow$ &&                   $\Downarrow$\\
$U$-$\theta$-closed &   $\Longrightarrow$   &     weakly $U$-$\theta$-closed \\
$\Downarrow$            &           & $\Downarrow$\\
$U$-closed &   $\Longrightarrow$  &     weakly $U$-closed\\
 $\Downarrow$ &&                        $\Downarrow$\\
 $\dots$ &$\dots$ &$\dots$\\
 $\Downarrow$ &&
$\Downarrow$\\
 $S(n-1)$-$\theta$-closed &   $\Longrightarrow$  &     weakly $S(n-1)$-$\theta$-closed\\
 $\Downarrow$       &&                 $\Downarrow$\\
$S(n-1)$-closed   & $\Longrightarrow$   &    weakly $S(n-1)$-closed\\
 $\Downarrow$   &&                        $\Downarrow$\\
$S(n)$-$\theta$-closed &   $\Longrightarrow$  &     weakly $S(n)$-$\theta$-closed\\
 $\Downarrow$   &&                        $\Downarrow$\\
$S(n)$-closed &   $\Longrightarrow$ &     weakly $S(n)$-closed\\
\end{tabular}

\bigskip

Note that all implications in the diagram are irreversible.
Examples of $S(n)$-closed but not $S(n)$-$\theta$-closed spaces
and $S(n)$-$\theta$-closed but not $S(n-1)$-closed spaces are
considered in  ~{\cite{dg}}. Examples are considered in
~{\cite{osip5}}, showing that the remaining implications are
irreversible.

\begin{theorem}

Let $X$ be a Lindel\"{e}of (finally compact) weakly $S(n)$-closed
$S(n)$-space. Then $X$ is $S(n)$-closed space.

\end{theorem}

\begin{proof}
Suppose the contrary. Let $X$ be a Lindel\"{e}of weakly
$S(n)$-closed but not $S(n)$-closed. Then in the spaces $X$ there
exists open filter $\mathcal F$ such that $ad_{\theta^n} \mathcal
F=\emptyset$. For each point $x\in X$ there are $F_{x}\in \mathcal
F$ and an $n$-hull $U_{x}$ of $F_{x}$ such that $x\notin
\overline{U_{x}}$. Note that $\bigcap_{x\in X}
\overline{U_{x}}=\emptyset$. Since $X$ is a Lindel\"{e}of, there
exists a countable family $\{U_{x_i}\}$ such that $\bigcap
\overline{U_{x_i}}=\emptyset$. Consider a sequence $\{y_{j}\}$
such that $y_{j}\in \bigcap_{i=1}^j F_{x_i}$. Clearly, the
infinite set $\{y_{j}\}$ has not a $\theta(n)$-complete
accumulation point. This contradicts the fact that $X$ is a weakly
$S(n)$-closed space.

\end{proof}

\begin{corollary}
Let $X$ be a countable weakly $S(n)$-$\theta$-closed $S(n)$-space.
Then $X$ is $S(n)$-closed.
\end{corollary}

\begin{corollary}
Let $X$ be second-countable weakly $S(n)$-$\theta$-closed
$S(n)$-space. Then $X$ is $S(n)$-closed.
\end{corollary}

Recall, that  a space is linearly Lindel\"{e}of (finally compact
in the sense of accumulation points) if every increasing open
cover $\{U_{\alpha} : \alpha\in \kappa \}$ has a countable
subcover (by increasing, we mean that $\alpha<\beta<\kappa$
implies $U_{\alpha}\subseteq U_{\beta}$).

\begin{theorem}

Let $n>1$ and $X$ be a linearly Lindel\"{e}of weakly
$S(n)$-$\theta$-closed $S(n)$-space. Then $X$ is weakly
$S(n-1)$-closed.
\end{theorem}

\begin{proof}
Suppose the contrary. Then there is a countable set $S$ such that
the set $S$ has not a $\theta(n-1)$-complete accumulation point.
Let $x\in X$ and $U$ be a $(n-1)$-hull of $x$ such that
$\overline{U}\bigcap S=\emptyset$. For every $y\in \overline{U}$
there exists a neighborhood $W_{y}$ of $y$ such that $W_{y}\bigcap
S=\emptyset$. Consider a open set $W=\bigcup_{y\in \overline{U}}
W_{y}$. Then $\overline{U}\subseteq W$ and $W$ is a $n$-hull of
the point $x$. Note that $W\bigcap S=\emptyset$. It is follows
that $x$ is not an $\theta^0(n)$-complete accumulation point of
$S$. This contradicts fact that $X$ is a weakly
$S(n)$-$\theta$-closed.
\end{proof}

\begin{corollary}
Let $n>1$ and $X$ be a Lindel\"{e}of weakly $S(n)$-$\theta$-closed
$S(n)$-space. Then $X$ is $S(n-1)$-closed.
\end{corollary}

\bigskip

 In ~{\cite{dg}} raised the question (Problem 5) about the product of $U$-$\theta$-closed spaces.
 Namely, it is required to prove or to disprove that the product of $U$-$\theta$-closed spaces is feebly compact.
 In particular, it was not known if every Lindel\"{e}of
 $U$-$\theta$-closed space is $H$-closed.

In ~{\cite{osip4}}, two Urysohn $U$-$\theta$-closed spaces whose
product is not feebly compact are constructed. Thus, the question
is negatively solved.

\begin{corollary}
Let $X$ be a Lindel\"{e}of $U$-$\theta$-closed Urysohn space. Then
$X$ is $H$-closed.
\end{corollary}

\begin{remark}
Observe that every $H$-closed space is feebly compact. By
corollary 3.11,  product of Lindel\"{e}of $U$-$\theta$-closed
spaces is feebly compact.
\end{remark}

\begin{corollary}
Let $n>1$ and $X$ be a Lindel\"{e}of weakly $S(n)$-$\theta$-closed
$S(n)$-space. Then $X$ is $S(n)$-$\theta$-closed.
\end{corollary}

Thus, for Lindel\"{e}of $S(n)$-spaces, classes of the considered
spaces are presented in the following diagram:

\bigskip

\begin{tabular}{ccc}
compact Hausdorff space &  $\Longleftrightarrow$ &     weakly
$S(1)$-$\theta$-closed\\
$\Downarrow$&&                                     $\Downarrow$\\
 $H$-closed &  $\Longleftrightarrow$ &     weakly $H$-closed\\
 $\Updownarrow$ &&                   $\Updownarrow$\\
$U$-$\theta$-closed &   $\Longleftrightarrow$   &     weakly $U$-$\theta$-closed \\
$\Downarrow$            &           & $\Downarrow$\\
$U$-closed &   $\Longleftrightarrow$  &     weakly $U$-closed\\
 $\Updownarrow$ &&                        $\Updownarrow$\\
 $\dots$ &$\dots$ &$\dots$\\
 $\Updownarrow$ &&
$\Updownarrow$\\
 $S(n-1)$-$\theta$-closed &   $\Longleftrightarrow$  &     weakly $S(n-1)$-$\theta$-closed\\
 $\Downarrow$       &&                 $\Downarrow$\\
$S(n-1)$-closed   & $\Longleftrightarrow$   &    weakly $S(n-1)$-closed\\
 $\Updownarrow$   &&                        $\Updownarrow$\\
$S(n)$-$\theta$-closed &   $\Longleftrightarrow$  &     weakly $S(n)$-$\theta$-closed\\
 $\Downarrow$   &&                        $\Downarrow$\\
$S(n)$-closed &   $\Longleftrightarrow$ &     weakly $S(n)$-closed\\
\end{tabular}

\bigskip
{\bf Question 1.} Does there exists a non $S(n)$-$\theta$-closed
 Lindel\"{e}of $S(n)$-closed space ($n>1$)?

\section{Characterizations $S(n)$-closed and $S(n)$-$\theta$-closed spaces}

Now for every $n\in \mathbb N$ we introduce an operator of
$\theta_0^n$-closure; for $M\subseteq X$ and $x\in X$ $x\notin
cl_{\theta_{0}^{n}} M$ if there is a $n$-hull $U$ of $x$ such that
$U\bigcap M=\emptyset$. A set $M\subseteq X$ is
$\theta_{0}^{n}$-closed if $M=cl_{\theta_{0}^{n}} M$.

\begin{definition}
A subset $M$ of a topological space $X$ is an
$S(n)$-$\theta_0^n$-set if every $S(n)$-cover $\gamma$ with
respect to $M$ ($M\subseteq \bigcup \{ Int_{\theta^n} U_{\alpha}:
U_{\alpha}\in \gamma\}$) by open sets of $X$ has a finite
subfamily which covers $M$ with the $\theta_0^n$-closures of its
members.
\end{definition}

\begin{definition}
The set $A$ is weakly $\theta(n)$-converge to the set $B$ if for
any $S(n-1)$-cover $\gamma=\{U_{\alpha}\}$ of $B$ there exists a
finite family $\{ U_{\alpha_i}\}_{i=1}^k \subseteq \gamma$ such
that $|A\setminus Int(\bigcup_{i=1}^k
\overline{U_{\alpha_i}})|<|A|$.
\end{definition}

\begin{theorem}
\label{tm3}
 For $n\in \mathbb N$, a $S(n)$-space $X$ is
$S(n)$-closed if and only if any infinity set $A\subseteq X$
weakly $\theta(n)$-converge to the set $B$ of its
$\theta(n)$-complete accumulation points.
\end{theorem}

\begin{proof}
Necessary. Let $X$ be $S(n)$-closed space and $A\subseteq X$. Take
any $S(n-1)$-cover $\gamma$ of $B$ where $B$ is the set of
$\theta(n)$-complete accumulation points of $A$. For each point
$x\notin B$ we take an $n$-hull $O(x)$ such that
$|\overline{O(x)}\bigcap A|<|A|$. Then we have an $S(n-1)$-cover
$\gamma'=\gamma\bigcup \{O(x) : x\notin B\}$ of $X$. As the space
$X$ is $S(n)$-closed there are finite families
$\{U_{i}\}_{i=1}^s\subseteq \gamma$ and $\{O(x_j)\}_{j=1}^{k}$
such that $(\bigcup_{i=1}^s \overline{U_{i}})\bigcup
(\bigcup_{j=1}^k \overline{O(x_{j})})=X$. Note that $A\setminus
Int(\bigcup_{i=1}^s \overline{U_{i}})\subseteq \bigcup_{j=1}^{k}
\overline{O(x_{j})}$.
 As $|A\bigcap(\bigcup_{j=1}^k \overline{O(x_{j}})|<|A|$ we have
$|A\setminus Int(\bigcup_{i=1}^s \overline{U_{i}})|<|A|$. Thus $A$
weakly $\theta(n)$-converge to the set $B$.

Note that $B$ is an $S(n)$-$\theta_0^n$-set. Really, $(A\bigcap
Int(\bigcup_{i=1}^s \overline{U_{i}}))\bigcap \overline{S(x)}\neq
\emptyset$ for every $x\in B$ and for any $n$-hull $S(x)$ of the
point $x$. It is follows that $S(x)\bigcap (\bigcup_{i=1}^s
U_{i})\neq \emptyset$ and $x$ is contained  in
$\theta_0^n$-closure of $\bigcup_{i=1}^s U_{i}$. Thus $B\subseteq
cl_{\theta_{0}^{n}} \bigcup_{i=1}^s U_{i}$.

Sufficiency. Let $\varphi=\{V_{\alpha}\}$ be open $S(n)$-filter on
$X$. Assume that $ad$ $\varphi=\emptyset$. Choose $V_{0}\in
\varphi$ such that $|V_0|=inf\{|V_{\alpha}|: V_{\alpha}\in \varphi
\}$. Since $\bigcap \overline{V_{\alpha}}=\emptyset$ we have
$\xi=\{U_{\alpha}: U_{\alpha}=X\setminus \overline{V_{\alpha}}\}$
is an $S(n-1)$-cover of $B$ where $B$ is the set of
$\theta(n)$-complete accumulation points of $V_0$. By condition,
there exists a finite family $\{ U_{\alpha_i}\}_{i=1}^k \subseteq
\xi$ such that $|V_0\setminus Int(\bigcup_{i=1}^k
\overline{U_{\alpha_i}})|<|V_0|$. Consider $V_{\alpha_i}\in
\varphi$ such that $U_{\alpha_i}=X\setminus V_{\alpha_i}$. Let
$V=\bigcap_{i=1}^k V_{\alpha_i}$ then $V\bigcap V_0\subseteq
V_0\setminus Int(\bigcup_{i=1}^k \overline{U_{\alpha_i}})$ and
$|V\bigcap V_0|<|V_0|$. This contradicts our choice of $V_{0}$.
Thus $X$ is $S(n)$-closed space.

\end{proof}

\begin{corollary}
Let $X$ be a $S(n)$-closed space and $A$ be an infinity set of
$X$. Then a set $B$ of $\theta(n)$-complete accumulation points of
$A$ is an $S(n)$-$\theta_0^n$-set.
\end{corollary}

\begin{definition}
The set $A$ is  $\theta^{0}(n)$-converge to the set $B$ if for any
$S(n-1)$-cover $\gamma=\{U_{\alpha}\}$ of $B$ there exists a
finite family $\{ U_{\alpha_i}\}_{i=1}^k \subseteq \gamma$ such
that $|A\setminus \bigcup_{i=1}^{k} U_{\alpha_i}|<|A|$.
\end{definition}

\begin{theorem}
For $n\in \mathbb N$, a $S(n)$-space $X$ is $S(n)$-$\theta$-closed
if and only if any infinity set $A\subseteq X$
$\theta^{0}(n)$-converge to the set $B$ of its
$\theta^{0}(n)$-complete accumulation points.
\end{theorem}

\section{Characterization minimal $S(n)$-spaces}

A $\mathcal P$ space is minimal $\mathcal P$ if it has no strictly
coarser $\mathcal P$ topology. The terms minimal Urysohn and
minimal regular are abbreviated as $MU$ and $MR$, respectively.

\begin{definition}
The set $A$ is $\theta(n)$-converge to the set $B$ if for any
$S(n-1)$-cover $\gamma=\{U_{\alpha}\}$ of $B$ there exists a
finite family $\{ U_{\alpha_i}\}_{i=1}^k \subseteq \gamma$ such
that $|A\setminus \bigcup_{i=1}^k \overline{U_{\alpha_i}})|<|A|$.
\end{definition}

\begin{theorem}\label{tm4}
For $n\in \mathbb N$, a $S(n)$-space $X$ is minimal $S(n)$-space
if and only if any infinity set $A\subseteq X$
$\theta(n)$-converge to the set $B$ of its $\theta(n)$-complete
accumulation points, and if there exists a point $x$ such that $A$
does not $\theta(n)$-converge to $X\setminus \{x\}$, then $x$ is a
complete accumulation point of $A$.
\end{theorem}

\begin{proof}
Necessary. Let $X$ be minimal $S(n)$-space and $A\subseteq X$.
Then $X$ is an $S(n)$-closed space (Corollary 2.3. in
~{\cite{jrw}}) and $A$ (weakly) $\theta(n)$-converge to the set
$B$ of its $\theta(n)$-complete accumulation points. Let $x\in X$
such that $A$ does not $\theta(n)$-converge to $X\setminus \{x\}$.
Consider $S(n-1)$-cover $\gamma=\{U_{\alpha}\}$ of $X\setminus
\{x\}$ such that $|A\setminus \bigcup_{i=1}^k
\overline{U_{\alpha_i}}|=|A|$ holds for any
$\gamma'=\{U_{\alpha_i}\}_{i=1}^k\subseteq \gamma$.

Let $\omega$ open $S(n)$-filter generated by $\{X\setminus
\overline{U_{\alpha}}: U_{\alpha}\in \gamma \}$. Then $\omega$ has
unique adherent point $x$. Since $X$ is minimal $S(n)$-space, we
have that open $S(n)$-filter $\omega$ converge to $x$. Thus for
every open neighborhood $O(x)$ of $x$ there is $V\in \omega$ such
that $V\subseteq O(x)$. So $|V\bigcap A|=|A|$ we have $x$ is a
complete accumulation point of $A$.

Sufficiency. We only need show that any open $S(n)$-filter
$\varphi$  with  unique adherent point $x$ is convergent.

Suppose that $ad$ $\varphi=\{x\}$, but $\varphi$ does not
converge. Then there is an open neighborhood $O(x)$ of $x$ such
that $W_{\alpha}=V_{\alpha}\setminus O(x)\neq \emptyset$, for any
$V_{\alpha}\in \varphi$. Choose $W_{\alpha_0}$ such that
$|W_{\alpha_0}|=inf\{|W_{\alpha}|: V_{\alpha}\in \varphi \}$. Let
$B$ be the set of $\theta(n)$-complete accumulation points of
$W_{\alpha_0}$. Note that $x\in B$. On a contrary, assume that
$x\notin B$ then for every $y\in B$ there are $n$-hull
neighborhood $O(y)$ of $y$ and $W_{\alpha}$ such that
$\overline{O(y)}\bigcap W_{\alpha}=\emptyset$. Consider
$S(n-1)$-cover $\gamma=\{O(y) : y\in B\}$ of $B$. For every finite
family $\{O(y_i)\}_{i=1}^k\subseteq \gamma$ there is $W_{\alpha}$
such that

$(\bigcup_{i=1}^k \overline{O(y_i)})\bigcap (W_{\alpha}\bigcap
W_{\alpha_0})=\emptyset$. By the choice of $W_{\alpha_0}$, we have
$|W_{\alpha}\bigcap W_{\alpha_0}|=|W_{\alpha_0}|$. Thus
$W_{\alpha_0}$ does not $\theta(n)$-converge to $B$. It follows
that $x\in B$ and $W_{\alpha_0}$ does not $\theta(n)$-converge to
$B\setminus\{x\}$. For each point $y\in X\setminus B$ we take an
$n$-hull $O_{1}(y)$ such that $|\overline{O_{1}(y)}\bigcap
A|<|A|$. Consider $S(n-1)$-cover $\gamma_1=\gamma\bigcup
\{O_{1}(y) : y\in X\setminus B\}$ of $X\setminus \{x\}$. For every
finite family $\{V_i\}_{i=1}^k\subseteq \gamma_1$  there is
$W_{\alpha}$ such that $(\bigcup_{i=1}^k \overline{V_i})\bigcap
(W_{\alpha}\bigcap W_{\alpha_0})=\emptyset$. Thus $W_{\alpha_0}$
does not $\theta(n)$-converge to $X\setminus \{x\}$.
 By the condition, $x$ is a complete
accumulation point of $W_{\alpha_0}$. This contradicts the fact
that $W_{\alpha_0}=V_{\alpha_0}\setminus O(x)$.

\end{proof}

Clearly, that the weakly $\theta(n)$-converge implies
$\theta(n)$-converge. By Theorems \ref{tm3} and \ref{tm4}, we have

\begin{theorem}\label{tm5}
For $n\in \mathbb N$, a $S(n)$-space $X$ is minimal $S(n)$-space
if and only if $X$ is a $S(n)$-closed, and if there exists a point
$x$ such that infinity set $A$ does not $\theta(n)$-converge to
$X\setminus \{x\}$, then $x$ is a complete accumulation point of
$A$.
\end{theorem}

In ~{\cite{fgpp}} raised the question (Q40) about the
characterization of $MU$ spaces.
 Namely, find a property $\mathcal Q$ which does not imply $U$-closed for which a space
  is $U$-closed and has property $\mathcal Q$ if and only if it is $MU$.
  The following theorem answers this question.

\begin{theorem}
An Urysohn space $X$ is $MU$ if and only if $X$ is a $U$-closed,
and if there exists a point $x$ such that infinity $A$ does not
$\theta(2)$-converge to $X\setminus \{x\}$, then $x$ is a complete
accumulation point of $A$.
\end{theorem}

In ~{\cite{fgpp}} raised the question (Q35): Does every $MU$ space
have a base of open sets with $U$-closed complements ?

Note, that the negative answer to this question is the following
example ~{\cite{herr}}. This is an example of a $MU$ space that
has no open base with $U$-closed complements.

\mbox{\bf Example 2.(Herrlich)}\label{ex1}\ \ For any ordinal
number $\alpha$, let $W(\alpha)$ be the set of all ordinals
strictly less than $\alpha$. Let $\omega_0$ be the first infinite
ordinal and $\omega_1$ the first uncountable ordinal. Let
$R=(W(\omega_{1}+1)\times W(\omega_{0}+1))\setminus \{(\omega_1,
\omega_0)\}$ and $R_n=R\times \{n\}$ where $n=0,\pm 1, \pm 2, ...$
Denote the elements of $R_n$ by $(x,y,n)$. Identify
$(\omega_1,y,n)$ with $(\omega_1,y,n+1)$ if $n$ is odd and
$(x,\omega_0,n)$ with $(x,\omega_0,n+1)$ if $n$ is even. Call the
resulting space $T$. To the subspace $E=R_1\bigcup R_2\bigcup R_3$
of $T$ add two points $a$ and $b$, and let $X=E\bigcup \{a,b\}$. A
set $V\subset X$ is open if and only if

(1) $V\bigcap E$ is open in $E$,

(2) $a\in V$ implies there exist $\alpha_0<\omega_0$ such that
$\{(\alpha,\beta,1) : \beta_0<\beta\leq \omega_0,
\alpha_0<\alpha<\omega_1\}\subset V$, and

(3) $b\in V$ implies there exist $\alpha_0<\omega_1$ and
$\beta_0<\omega_0$ such that $\{(\alpha,\beta,3) : \beta_0<\beta<
\omega_0, \alpha_0<\alpha\leq\omega_1\}\subset V$.

Really, for any open $V\ni a$, if $a\in O(a)=\{(\alpha,\beta,1) :
\beta_0<\beta\leq \omega_0, \alpha_0<\alpha<\omega_1\}\subset V$
then $X\setminus O(a)$ is not a $U$-closed. Infinity set
$\{(\alpha,\omega_0,2) :  \alpha_0<\alpha<\omega_1\}$ do not
 weakly $\theta(2)$-converge to the set of its
$\theta(2)$-complete accumulation points.

\section{Characterization $S(n)FC$ spaces}

\begin{definition}
A $S(n)$-space is $S(n)$-functionally compact ($S(n)FC$) if every
continuous function onto a $S(n)$-space is closed.
\end{definition}

A set $C$ will be called complete accumulation set of a set $A$ if
$|U\bigcap A|=|A|$ for any open set $U\supseteq C$.

\begin{theorem}\label{tm6}
For $n\in \mathbb N$, a $S(n)$-space $X$ is $S(n)FC$ if and only
if any infinity set $A\subseteq X$ $\theta(n)$-converge to the set
$B$ of its $\theta(n)$-complete accumulation points, and if there
exists a $\theta^n$-closed set $C$ such that $A$ does not
$\theta(n)$-converge to $X\setminus C$, then $C$ is a complete
accumulation set of $A$.
\end{theorem}

\begin{proof}
Necessary. Let $X$ be $S(n)FC$ and $A\subseteq X$. Then $X$ is an
$S(n)$-closed space and $A$ (weakly) $\theta(n)$-converge to the
set $B$ of its $\theta(n)$-complete accumulation points. Let
$\theta^n$-closed set $C$ such that $A$ does not
$\theta(n)$-converge to $X\setminus C$.

Consider $S(n-1)$-cover $\gamma=\{U_{\alpha}\}$ of $X\setminus C$
such that $|A\setminus \bigcup_{i=1}^k
\overline{U_{\alpha_i}}|=|A|$ holds for any
$\gamma'=\{U_{\alpha_i}\}_{i=1}^k\subseteq \gamma$.

Let $\omega$ open $S(n)$-filter generated by $\{X\setminus
\overline{U_{\alpha}}: U_{\alpha}\in \gamma \}$.

  Suppose that there exists an open set $W\supseteq C$ such that $|A \bigcap
  W|<|A|$. Consider the quotient space $(X/C,\tau)$ of $X$ with $C$
  identified to a point $c$. Now $\tau_1=\{ V\in \tau : c\in V$ implies $V\in
  \omega \}$ is a topology on $X/C$. In $(X/C, \tau_1)$ we have
  $ad_{\theta^n}N_{x}$ for any $x$ where $N_{x}$ is the neighbourhood filter at the point $x$,
   and thus $(X/C, \tau_1)$ is an $S(n)$-space. It is clear that $\tau_1$ is strictly coarser
  than $\tau$. The quotient function from $X$ to $X/C$ is denoted
  as $p_{C}$, and $q_{C}$ denotes $s\circ p_{C}$ where $s: (X/C,
  \tau) \rightarrow (X/C, \tau_1)$ is the identity function.
Note that $q_{C}(X\setminus W)$ is not closed in $(X/C, \tau_1)$.
Thus, $q_{C}$ is continuous but is not closed. This is a
contradiction that $X$ is a $S(n)FC$ space.

Sufficiency. Suppose that $X$ is not $S(n)FC$ space. Then there is
a continuous function $f$ from $X$ onto an $S(n)$-space $Y$ such
that $f$ is not closed. Consider the closed set $A\subseteq X$
such that $f(A)$ is not closed. Let $y\in \overline{f(A)}\setminus
f(A)$ and $N_{y}=\{V_{\alpha}\}$ is the neighbourhood
$S(n)$-filter at the point $y$. Then $B=f^{-1}(y)=\bigcap
\{f^{-1}(V_{\alpha})\}$ and $B$ is a $\theta^n$-closed subset of
$X$. Note that $X\setminus A$ is an open set containing $B$ such
that $W_{\alpha}=U_{\alpha}\setminus (X\setminus A)\neq \emptyset$
for any $U_{\alpha}\in \{f^{-1}(V_{\alpha})\}$.

 Choose $W_{\alpha_0}$ such that
$|W_{\alpha_0}|=\inf\limits_{\alpha} \{|W_{\alpha}|\}$.

Let $D$ be the set of $\theta(n)$-complete accumulation points of
$W_{\alpha_0}$. By the condition, set $W_{\alpha_0}$
$\theta(n)$-converge to the set $D$. We claim that
$\theta^n$-closed set $B$ such that $W_{\alpha_0}$ does not
$\theta(n)$-converge to $X\setminus B$. Indeed, for any $x\in
X\setminus B$ there is $U_{\alpha_{x}}$ such that $x$ is
$S(n)$-separated from $U_{\alpha_{x}}$. Let $O(x)$ be a $n$-hull
neighbourhood of $x$ such that $\overline{O(x)}\bigcap
U_{\alpha_{x}}=\emptyset$. Consider a $S(n-1)$-cover
$\gamma=\{O(x) : x\in X\setminus B \}$ of $X\setminus B$. For any
finite family $\{ O(x_i)\}_{i=1}^{k} \subseteq \gamma$ there is
$U=\bigcap\limits_{i=1}^{k} U_{\alpha_{x_i}}$ such that
$\bigcup\limits_{i=1}^{k} \overline{O(x_i)}\bigcap (U\bigcap
W_{\alpha_0})=\emptyset$. By the choice of $W_{\alpha_0}$, we have
$|U \bigcap W_{\alpha_0}|=|W_{\alpha_0}|$. It follows that
$W_{\alpha_0}$ does not $\theta(n)$-converge to $X\setminus B$. By
the condition, $B$ is a complete accumulation set of
$W_{\alpha_0}$. This contradicts the fact that $X\setminus A$ is
an open set containing $B$.

\end{proof}

\begin{corollary}
\label{tm7} An Urysohn space $X$ is $UFC$ if and only if $X$ is an
$U$-closed, and if there exists a $\theta^2$-closed set $C$ such
that infinity set $A$ does not $\theta(2)$-converge to $X\setminus
C$, then $C$ is a complete accumulation set of $A$.
\end{corollary}

\begin{definition}
A $S(n)$-space $X$ is $S(n)FFC$ ($S(n)CFC$) if every continuous
function $f$ onto a $S(n)$-space $Y$ with $f^{-1}(y)$ finite
(compact) is a closed function.
\end{definition}

\begin{theorem}\label{tm26}
For $n\in \mathbb N$, a $S(n)$-space $X$ is $S(n)FFC$ ($S(n)CFC$)
if and only if $X$ is a $S(n)$-closed, and if there exists a
finite (compact) set $C$ such that infinity set $A$ does not
$\theta(n)$-converge to $B\setminus C$, then $C$ is a complete
accumulation set of $A$.
\end{theorem}

\begin{proof} A proof of Theorem~\ref{tm26} is analogous to that of Theorem
~\ref{tm6}.
\end{proof}

{\bf Question 2.} Is every $S(n)FC$ ( $S(n)FFC$, $S(n)CFC$ ) space
necessarily compact  $(n>1)$?

\section{$S(\omega)$-closed and minimal $S(\omega)$ spaces}

 Two filters $\mathcal F$ and $\mathcal Q$ on a space $X$ are $S(\omega)$-separated
 if there are open families $\{ U_{\beta}: \beta<\omega\}\subseteq
 \mathcal F\}$ and $\{ V_{\beta}: \beta<\omega\}\subseteq
 \mathcal Q\}$ such that $U_{0}\bigcap V_{0}=\emptyset$ and for
 $\gamma+1<\omega$, $clU_{\gamma+1}\subseteq U_{\gamma}$ and
 $clV_{\gamma+1}\subseteq V_{\gamma}$. A space $X$ is $S(\omega)$
 if for distinct points $x,y\in X$, the neighborhood filters
 $\mathcal N_{x}$ and $\mathcal N_{y}$ are $S(\omega)$-separated.

A $S(\omega)$-closed space is a $S(\omega)$ space closed in any
$S(\omega)$ space containing them.

In 1973, Porter and Votaw ~{\cite{pv}} established next results.

(1) A minimal $S(\omega)$ space is $S(\omega)$-closed and
semiregular.

(2) A minimal $S(\omega)$ space is regular.

(3) A space is $R$-closed if and only if it is $S(\omega)$-closed
and regular.

(4) A space is $MR$ if and only if it is minimal $S(\omega)$.

\begin{definition} A neighborhood $U$ of a point $x$ is called an $\omega$-hull of
the point $x$ if there exists a set of neighborhoods
$\{U_{i}\}_{i=1}^{\infty}$ of the point $x$ such that
$clU_{i}\subseteq U_{i+1}$ and $U_{i}\subseteq U$ for every $i\in
\mathbb N$.
\end{definition}

\begin{definition} A point $x$ from $X$ is called a
$\theta(\omega)$-accumulation point of an infinite set $F$ if
$|F\bigcap U|=|F|$ for any $U$, where $U$ is an $\omega$-hull of
the point $x$.
\end{definition}

\begin{definition}
The set $A$ is $\theta(\omega)$-converge to the set $B$ if for any
regular cover $\gamma=\{U_{\alpha}\}$ of $B$ there exists a finite
family $\{ U_{\alpha_i}\}_{i=1}^s \subseteq \gamma$ such that
$|A\setminus \bigcup_{i=1}^s U_{\alpha_i}|<|A|$.
\end{definition}

\begin{theorem}\label{tm10}
A regular space $X$ is $R$-closed if and only if any infinity set
$A\subseteq X$ $\theta(\omega)$-converge to the set $B$ of its
$\theta(\omega)$-accumulation points.
\end{theorem}

\begin{proof}A proof of Theorem~\ref{tm10} is analogous to that of Theorem
~\ref{tm3}.
\end{proof}

In ~{\cite{fgpp}} raised the question (Q39) about the
characterization of $MR$ spaces.
 Namely, find a property $\mathcal P$ which does not imply $R$-closed for which a space
  is $R$-closed and has property $\mathcal P$ if and only if it is $MR$.
  The following theorem answers this question.

\begin{theorem}\label{tm11}
A regular space $X$ is $MR$ space if and only if $X$ is a
$R$-closed, and if there exists a point $x\in B$ such that
infinite set $A$ does not $\theta(\omega)$-converge to $X\setminus
\{x\}$, then $x$ is a complete accumulation point of $A$.
\end{theorem}

\begin{proof} A proof of Theorem~\ref{tm11} is analogous to that of Theorem
~\ref{tm4}.
\end{proof}

We introduce an operator of $\theta^{\omega}$-closure; for
$M\subseteq X$ and $x\in X$ $x\notin cl_{\theta^{\omega}} M$ if
there is a $\omega$-hull $U$ of $x$ such that $U\bigcap
M=\emptyset$. A set $M\subseteq X$ is $\theta^{\omega}$-closed if
$M=cl_{\theta^{\omega}} M$.

\begin{definition}
A regular space is  regular functionally compact ($RFC$) if every
continuous function onto a  regular space is closed.
\end{definition}

\begin{theorem}\label{tm16}
A regular space $X$ is $RFC$ if and only if $X$ is $R$-closed, and
if there exists a $\theta^{\omega}$-closed set $C$ such that
infinite set $A$ does not $\theta(\omega)$-converge to $X\setminus
C$, then $C$ is a complete accumulation set of $A$.
\end{theorem}

\begin{proof} A proof of Theorem~\ref{tm16} is analogous to that of Theorem
~\ref{tm6}.
\end{proof}

{\bf Question 3.} Is every $RFC$ space necessarily compact ?

\bigskip

\begin{remark} Note that a negative answer to the question (Q27 in~{\cite{fgpp}}) of
compactness of $U$-closed space in which the closure of any open
set is $U$-closed be any non-compact $H$-closed Urysohn space.
\end{remark}

\bigskip

{\bf Question 4.}(Q26 in~{\cite{fgpp}}) Is an $R$-closed space in
which the closure of every open set is $R$-closed necessarily
compact ?

\bigskip

This work was supported by the Russian Foundation for Basic
Research (project no.~09-01-00139-a)  and by the Division of
Mathematical Sciences of the Russian Academy of Sciences (project
no.~09-T-1-1004).

\bibliographystyle{plain}

\end{document}